\documentclass[12pt]{article}

\usepackage[english]{babel}
\usepackage{geometry}
\usepackage[colorlinks,citecolor=blue,urlcolor=blue]{hyperref}
\usepackage[T1]{fontenc}
\usepackage{multirow}
\usepackage{xspace}
\usepackage{color}
\usepackage{graphicx}
\usepackage{booktabs}
\usepackage{comment}
\usepackage{algorithm}
\usepackage{algpseudocode}
\input{main.sty}

\usepackage{amsthm}
\newtheorem{theorem}{Theorem}
\newtheorem{lemma}{Lemma}
\newtheorem{remark}{Remark}

\newtheorem{proposition}{Proposition}

\newenvironment{keywords}
    {\textbf{Keywords:}}
    {}

\newenvironment{MSC}
    {\textbf{\href{https://mathscinet.ams.org/mathscinet/msc/msc2020.html}{MSC2020}:}}
    {}

\usepackage{authblk}
\renewcommand*{\Affilfont}{\normalsize\normalfont}

\title{Multilayer hypergraph clustering\\ using the aggregate similarity matrix\footnote{Supported by the Vilho, Yrjö and Kalle Väisälä Foundation of the Finnish Academy of Science and Letters, and the French government through the RISE Academy of UCA$^\text{JEDI}$ Investments in the Future project managed by the National Research Agency (ANR) with the reference number ANR-15-IDEX-0001. This is an updated author version of the paper accepted to
Proceedings of the 18th Workshop on Algorithms and Models for the Web Graph (WAW 2023)
published by Springer. This version has been updated to incorporate the possibility of non-uniform layers.}}
\author[1]{Kalle Alaluusua}
\author[2]{Konstantin Avrachenkov}
\author[2]{B. R. Vinay Kumar}
\author[1]{Lasse Leskel\"{a}}
\affil[1]{ Aalto University, Espoo, Finland}
\affil[ ]{\Affilfont \{\href{mailto:kalle.alaluusua@aalto.fi}{kalle.alaluusua}, \href{mailto:lasse.leskela@aalto.fi}{lasse.leskela}\}@aalto.fi}
\affil[2]{ INRIA, Sophia Antipolis, Valbonne, France}
\affil[ ]{\Affilfont \{\href{mailto:vinay-kumar.bindiganavile-ramadas@inria.fr}{vinay-kumar.bindiganavile-ramadas}, \href{mailto:k.avrachenkov@inria.fr}{k.avrachenkov}\}@inria.fr}
\date{March 17, 2023}
\renewcommand\Affilfont{\itshape\small}

\begin{document}

\maketitle
\thispagestyle{plain}
\pagestyle{plain}

\begin{abstract}
We consider the community recovery problem on a multilayer variant of the hypergraph stochastic block model (HSBM). Each layer is associated with an independent realization of a $d$-uniform HSBM on $N$ vertices. Given the similarity matrix containing the aggregated number of hyperedges incident to each pair of vertices, the goal is to obtain a partition of the $N$ vertices into disjoint communities. In this work, we investigate a semidefinite programming (SDP) approach and obtain information--theoretic conditions on the model parameters that guarantee exact recovery both in the assortative and the disassortative cases. 

\vspace{12pt} \noindent \begin{keywords}
hypergraph SBM, community detection, semidefinite programming, multilayer, clustering, planted partition
\end{keywords}

\noindent \begin{MSC}
05C65, 05C80, 62H30, 90B15, 90C22, 90C35, 94A16
\end{MSC}

\end{abstract}

\section{Introduction}

Traditional network data are observed as interactions between node pairs, represented as a graph, or equivalently as an adjacency matrix.  More refined forms of network data may involve multiple types of higher-order interactions simultaneously involving multiple nodes.  Such a data set can be viewed as a binary array $A_e^{(m)}$ indexed by node sets $e$ and positive integers $m$ so that $A_e^{(m)}=1$ indicates that an interaction of type $m$ occurs among node set $e$. 
The array can also be viewed as a multilayer hypergraph where the entries of the array indicate the presence of hyperedge $e$ in the $m$-th layer.  
Further, the entries could also depend on the communities of the underlying nodes. Stochastic block models (SBMs) are a popular choice for generative models with a community structure for such applications. Hypergraph stochastic block models (HSBMs) introduce hyperedges into SBMs, thus extending their modelling capabilities. In the following, we provide two examples to illustrate the relevance of the model discussed in this work
\begin{enumerate}
    \item \textbf{Table reservations at restaurants}: Consider $M$ restaurants each of which have tables that can accommodate $d$ people. A population of $N$ individuals is divided into two different communities; those who prefer vegan and those who do not. A hyperedge links $d$ individuals with one of the $M$ restaurants indicating that the corresponding $d$ individuals made a reservation at the restaurant at some point of time. Thus, each of the $M$ restaurants correspond to a layer. Naturally, individuals who prefer vegan dishes gravitate towards those restaurants which have more vegan options. But since the choice of a restaurant is made by a group of $d$ individuals who can belong to either community, the probability that a restaurant is visited will depend on the number of individuals in each of the communities. In other words, a subset of $d$ individuals form a hyperedge which is present with probability dependent on the communities of the $d$ individuals.
    \item \textbf{Processor sharing}: $N$ tasks are assigned to $M$ heterogeneous servers each of which has several $d$-core processors that can process $d$ tasks at a time. A subset of $d$ tasks constitute a hyperedge, and is assigned to a processor based on the priorities (high or low) of each of the tasks and the computation capability of the server.
\end{enumerate}

We will next describe a generative model of a hypergraph with $N \ge 1$ nodes and $M \ge 1$ layers, where each hyperedge in layer $m$ has size $d_m \ge 2$. We take $M$ as a constant independent of $N$. The set of nodes is divided into two communities of equal sizes (we assume $N$ is even), and the resulting community structure, denoted by $\bsigma^{(N)}$, is uniformly distributed on the set $\{(\sigma_1,\sigma_2,\cdots,\sigma_N)\in \{-1,+1\}^N : |\{i:\sigma_i=+1\}| = |\{i:\sigma_i=-1\}| \}$.
The community profile of a node set $e$ is defined as a vector $\btau(e) = (\tau_{-1}(e), \tau_{+1}(e))$, where $\tau_k(e)$ is equal to the number of nodes in $e$ with community membership $k$.  We will then sample a multilayer hypergraph on node set $[N] = \{1,\dots,N\}$ so that each node set $e \subset [N]$ of layer $m$ having size $d_m$ and community profile $\btau(e) = \bm{t}$ is linked by a hyperedge in layer $m$ with probability
\begin{align}
 p_{\bm{t}}^{(m)}
 \weq \frac{\alpha_{\bm{t}}^{(m)} \log N}{\binom{N-1}{d_m-1}}\wedge 1,
 \label{eq:regime}
\end{align}
independently of other node sets and layers. Here $\alpha_{\bt}^{(m)}>0$ is a real number independent of $N$. This scaling of the hyperedge probabilities ensures that the expected average degree of each node is $\Theta(\log N)$. References \cite{ahn_lee_suh_2018,chien_lin_wang_2019,Kim_Bandeira_Goemans_2018-07-08} show that the phase-transition for exact recovery occurs in this regime, and this regime is also critical for connectivity in general hypergraph models \cite{Bergman_Leskela_2022}.

The non-uniform multilayer hypergraph can be represented as a binary array $\bA = (A_e^{(m)})$ in which the entries are mutually independent Bernoulli random variables.  The event
$\{A_e^{(m)}=1\}$ has probability $p_{\bm{t}}^{(m)}$ when $\btau(e) = \bt$. To indicate that the pair
$(\bsigma, \bA) = (\bsigma^{(N)}, \bA^{(N)})$ is sampled from the model, we abbreviate $(\bsigma, \bA) \sim \hsbm(N,M,(d_m),(\alpha_{\bt}^{(m)}))$.
We will focus on a \emph{symmetric} model in which
\[
 \alpha_{(r,d_m-r)}^{(m)}
 \weq \alpha_{(d_m-r,r)}^{(m)}
\]
for all $0 \le r \le d_m$ and all $m$. This means that the presence of the hyperedge depends only on the number of nodes of each community rather than the community label.

The problem of community detection is to output an estimate $\bm{\hat{\sigma}}^{(N)}$ of the underlying node communities. The estimate is said to achieve \emph{exact recovery} if, 
\begin{align}
    \lim_{N\rightarrow \infty} \P\left(\bm{\hat{\sigma}}^{(N)} \in \{\pm \bsigma^{(N)}\}\right) = 1.
\label{def:exact_recovery}
\end{align}
where $\P$ is the probability measure associated with $\hsbm(N,M,d_m,(\alpha_{\bt}^{(m)}))$, a generative model for the community structure and the observations.

In this work, the main focus is to study the community recovery problem based on a layer-aggregated similarity matrix $W_{ij} = \sum_m W_{ij}^{(m)}$ where $ (W_{ij}^{(m)}) =: \bW^{(m)}$ is a zero-diagonal matrix with off-diagonal entries
\[
 W_{ij}^{(m)}
 \weq \sum_{e: e \ni i,j} A^{(m)}_e
\]
counting the number of hyperedges in layer $m$ incident to nodes $i$ and $j$.
Community recovery based on the similarity matrix $\bW$ instead of the full data set $\bA$ is motivated by two aspects: privacy and computational tractability.
For example, in the application of table reservations at restaurants, providing the full hypergraph could reveal the frequency a particular individual visits a restaurant. This could violate the privacy of the individual. On the other hand, providing the similarity matrix obfuscates such individual information, since information regarding the restaurants that are visited is not revealed. Additionally, the similarity matrix provides a compact matrix representation of the hypergraph that is easier to manipulate using matrix algebra. Nevertheless, it is clear that the similarity matrix contains less information than the complete hypergraph. 

In this work, we investigate a semidefinite programming (SDP) approach for solving the community detection problem using $\bW$. Our main result in Section~\ref{sec:sdp} provides an information quantity that characterizes the performance of the SDP approach. To be more specific, it gives a sufficient condition relating the different parameters of the model for exact recovery using the SDP technique.  

The paper is organized as follows. In Section~\ref{sec:related}, we discuss recent literature in the area and highlight the key differences in our work. In Section~\ref{sec:sdp}, we describe the semidefinite programming algorithm and state our main result concerning the information theoretic conditions on the parameters of the model that guarantee exact recovery. Section \ref{sec:numerical} contains some simulation results of our algorithm on synthetic data generated using the HSBM model. The proofs of our results are provided in Section~\ref{sec:analysis}. Section~\ref{sec:conclusion} concludes the paper and provides some future directions.

\section{Related work}
\label{sec:related}

The usual paradigm for theoretical research in community detection is to first propose a generative model that captures the application (data) as a graph or a network, followed by the analysis of a clustering algorithm for the proposed model. Community recovery algorithms on stochastic block models (SBMs) have attracted considerable attention in the past. A comprehensive survey of the field is provided in \cite{Abbe_2018_JMLR}, see also a review of recent results on graph clustering in \cite{Avrachenkov_Dreveton_2022}. Our interest in this work is on multilayer hypergraphs. We first provide a brief survey of recent work in clustering on multilayer networks followed by hypergraph SBMs.

A seminal work for multilayer networks is the review article \cite{Kivela_etal_2014}. Subsequently, in \cite{pensky_zhang_2019}, the authors consider estimating the membership for each individual layer in a multilayer SBM. In the special case that memberships do not change, their method works on the normalized aggregate adjacency matrix. The authors in \cite{lei_chen_lynch_2020} establish that increasing the number of network layers guarantees consistent community recovery (by a least squares estimator) even when the graphs get sparser. SBMs with general interactions allow an alternate model for multilayer networks. These are studied in \cite{Avrachenkov_Dreveton_Leskela_2022} where the authors address the community recovery problem using aggregate adjacency matrix as well as the full graph. The authors in \cite{Alaluusua_Leskela_2022} study Bayesian community recovery in a regime where both the number of nodes and the number of layers increase. 
 
With regard to literature on hypergraphs, the hypergraph stochastic block model (HSBM) was first introduced by Ghoshdastidar and Dukkipati in \cite{Ghoshdastidar_Dukkipati_2014} to capture higher order interactions. They also show strong consistency of spectral methods in dense uniform hypergraphs. In subsequent works \cite{Ghoshdastidar_Dukkipati_2015_ICML,ghoshdastidar_dukkipati_2015_AAAI,Ghoshdastidar_Dukkipati_2017}, they extend their results to partial recovery in sparse non-uniform hypergraphs. Some other works on the spectral algorithm for hypergraphs include \cite{ahn_lee_suh_2018,Pal_Zhu_2021}.

The recent work by Zhang and Tan \cite{Zhang_Tan_2023}  considers the general $d$-uniform HSBM with multiple communities. They establish a sharp phase transition for exact recovery when the knowledge of the whole hypergraph is given. They recover results from several previous works including \cite{chien2018community,chien_lin_wang_2019,Kim_Bandeira_Goemans_2018-07-08,Pal_Zhu_2021}. In the process of solving the exact recovery problem, they do show almost exact recovery using only the similarity matrix through a spectral approach. Another general hypergraph model with theoretical guarantees by \cite{zhen_wang_2022} employs a latent space representation of nodes to cover HSBM, non-uniform hypergraphs, and hypergraphs with heterogeneity among nodes.

Some of the other approaches used in the literature for the community recovery problem on hypergraphs are based on modularity \cite{Kamiski_Poulin_Pralat_Szufel_Theberge_2019,Kaminski_Pralat_Theberge_2021,Kaminski_Pralat_Theberge_2022-10-26,kumar2019new,Kumar_et_al_2020}, tensor decomposition \cite{ke2019community,zhen_wang_2022}, random walk based methods \cite{Pal_Zhu_2021,stephan_Zhu_2022,Zhou_et_al_2006}, variational inference \cite{Brusa_Matias_2022}, and approximate message passing \cite{angelini_caltagirone_krzakala_zdeborova_2015,lesieur2017statistical}.

In this work, we investigate the problem of exact recovery on the HSBM through the lens of semidefinite programming (SDP).
Our work is closest in spirit to \cite{Gaudio_Joshi_2022-08-23} and \cite{Kim_Bandeira_Goemans_2018-07-08} that discuss the SDP approach. The SDP formulation (described in Section \ref{sec:sdp}) arises as a relaxation of the computationally hard procedure of finding a nodes' partition with minimum number of edges crossing it. In \cite{Kim_Bandeira_Goemans_2018-07-08}, the authors show that for a $d$-uniform homogeneous HSBM with two equal-sized and symmetric communities, exact recovery using the full hypergraph shows a sharp phase transition behavior. They go on to propose a `truncate-and-relax' algorithm that utilizes the structure of the similarity matrix. An SDP approach then guarantees exact recovery with high probability, albeit in a parameter regime which is slightly sub-optimal. This gap is bridged in \cite{Gaudio_Joshi_2022-08-23} who consider the community recovery problem with the knowledge of only the similarity matrix. Below, we highlight the differences from these previous works:
\begin{enumerate}
    \item The authors in both \cite{Gaudio_Joshi_2022-08-23} and \cite{Kim_Bandeira_Goemans_2018-07-08} consider the homogeneous model in which the hyperedge parameters take just two values corresponding to all nodes of a hyperedge being in the same community, and at least one of them being in a different community. Related works with the same assumption include \cite{ahn_lee_suh_2018,cole_zhu_2020,lee_kim_chung_2020}. In this work, we allow for hyperedge parameters to depend on the number of nodes of each community in the hyperedge resulting in an inhomogeneous HSBM as in \cite{Zhang_Tan_2023}, albeit with the symmetric assumption. A similar assumption is present in other works such as \cite{Ghoshdastidar_Dukkipati_2014,ghoshdastidar_dukkipati_2015_AAAI} as well.

    \item Much of the earlier work assumes that the data is 
     \textit{assortative} or \textit{homophilic}, i.e. nodes in the same community are more likely to be adjacent to each other than to nodes in different communities.
     Our results incorporate the \textit{disassortative} or \textit{heterophilic} case where the opposite is true. This could be of interest for some applications: reputation of a research institute is partly assessed based on the amount of collaboration with experts from external institutions (see e.g. \cite{guerrero_2020}); so, one might expect certain research networks to be disassortative.

    \item Our model targets multilayer HSBMs that can be seen as a generalization of previous models. Moreover, these layers could individually be assortative or disassortative which could then capture a plethora of applications.
\end{enumerate}

\section{Algorithm and main results}
\label{sec:sdp}
A first approach to obtain an estimate of the node communities given the similarity matrix $W$ is to solve the min-bisection problem:
\begin{align}
    \max \sum_{i,j} W_{ij} x_ix_j \hspace{0.5cm}\text{ subject to } \bx\in \{\pm 1\}^N, \1^T\bx = 0.
    \label{alg:min-bisection}
\end{align}
This formulation assumes that the data is assortative. In the disassortative case the opposite is true, and we replace the maximization in \eqref{alg:min-bisection} with minimization, or equivalently change the sign of the objective function. However, the min-bisection problem is known to be NP-hard in general (see \cite{kim2017community}), which is why \cite{Kim_Bandeira_Goemans_2018-07-08} introduces a semidefinite programming (SDP) relaxation of \eqref{alg:min-bisection}. Algorithm \ref{alg:sdp} introduces an additional input $s \in \{\pm 1\}$ and generalizes their relaxation to both assortative ($s = +1$) and disassortative ($s = -1$) cases.

\begin{algorithm}
\algrenewcommand\algorithmicrequire{\textbf{Input:}}
\algrenewcommand\algorithmicensure{\textbf{Output:}}
\caption{ }
\begin{algorithmic}[1]
\Require $N\times N$ similarity matrix $\bW$ and $s \in \{\pm 1\}$.
\Ensure Community estimate $\bm{\hat{\sigma}}$\\
Solve the following optimization problem:
\begin{align}
    \begin{split}
    \text{maximize}
    &\sum_{0 \leq i < j\le N} s W_{ij}X_{ij}\\
    \text{subject to} 
    & \sum_{0 \leq i \leq j\le N} X_{ij} = 0,\\
    & X_{ii} = 1 \text{ for all } i \in [N]\\
    & \bX \succeq 0.
    \end{split}
    \label{eq:sdp}
\end{align}\\
Let $\bX^*\in \R^{N\times N}$ be the optimal solution of \eqref{eq:sdp} and let it have an eigendecomposition $\bX^* = \sum_{i = 1}^N \lambda_i \bm{v_i}\bm{v_i}^\text{T}$ with $\lambda_1 \geq \lambda_2 \geq \dots \geq \lambda_N$.\\
Output $\bm{\hat\sigma} = \operatorname{sgn}(\bm{v_1})$
\end{algorithmic}
\label{alg:sdp}
\end{algorithm}

\begin{remark}
An alternate approach in the disassortative case is to consider the complement of the hypergraph, which is assortative. A similarity matrix of the complement is given by $\binom{N - 2}{d - 2} (\bm 1 \bm 1^{\text{T}} - I) - \bW$. However, owing to our scaling assumption in \eqref{eq:regime}, the resulting similarity matrix of the complement is no longer in the same regime, and requires a different approach to analyze.
\end{remark}

To capture the level of assortativity of our model, we define the following quantity, accordingly referred to as the \emph{assortativity}
\begin{equation}
 \label{eq:Assortativity}
 \xi \ \coloneqq \ \sum_{m=1}^M 2^{-(d_m-1)} \sum_{r = 0}^{d_m-1} \binom{d_m-1}{r}(d_m-1-2r)\alpha_{(r, d_m - r)}^{(m)}.
\end{equation}
The summation over the different layers implies that the full hypergraph can be assortative even if individual layers $\bW^{(m)}$ are not. Table~\ref{tab:Assortativity} specifies Formula \eqref{eq:Assortativity} for selected values of $d \coloneqq d_m$. The following proposition (proved in Section \ref{sec:assortativity}) states that $\xi$ is a normalized expected difference between the number of hyperedges shared between two nodes when they are of the same community and when they are of different communities.
\begin{proposition}
\label{prop:assortativity_W}
For $i \neq j$, let $w_{\rm in} = \E[W_{ij}|\sigma_i = \sigma_j]$ and $w_{\rm out} = \E[W_{ij}|\sigma_i \neq \sigma_j]$. Then,
$$w_{\rm{in}} - w_{\rm{out}} = \frac{\log N}{N} \xi  + o\left(\frac{\log N}{N}\right).$$
\end{proposition}
Therefore, a model is said to be assortative if $\xi > 0$, and disassortative if $\xi < 0$. 

\begin{table}
\caption{\label{tab:Assortativity} Assortativity $\xi$ in a $d$-uniform HSBM with $M$ layers, where we denote $\alpha_{\bt} = \sum_{m=1}^M \alpha_{\bm{t}}^{(m)}$.}
\centering
\small
\begin{tabular}{ll}
\toprule
$d$ \hspace{2ex} & $\xi$ \\
\midrule
2 & $\alpha_{(0, 2)} - \alpha_{(1,1)}$  \\
3 & $2 \alpha_{(0, 3)} - 2 \alpha_{(1, 2)}$ \\
4 & $3 \alpha_{(0, 4)} - 3 \alpha_{(2, 2)} $ \\
5 & $4 \alpha_{(0, 5)} + 4\alpha_{(1, 4)} - 8 \alpha_{(2,3)} $ \\
6 & $5\alpha_{(0, 6)} + 10\alpha_{(1, 5)} - 5\alpha_{(2, 4)} - 10\alpha_{(3, 3)}$ \\
\bottomrule
\end{tabular}
\end{table}

In order to state our main result concerning the performance of Algorithm~\ref{alg:sdp}, we will need the following information quantity%
\begin{equation}
 \label{eq:info_metric_layers}
I
 \weq \ \sup_{\lambda \in \R} \sum_{m=1}^M \sum_{r=0}^{d_m-1}
 2^{-(d_m-1)} \binom{d_m-1}{r} \alpha_{(r, d_m-r)}^{(m)}
 \left( 1-e^{-\lambda\left(d_m - 1 - 2r\right)} \right),
\end{equation}
which captures the propensity of a node to be present in a hyperedge containing more nodes of the same (different) community as itself in the assortative (disassortative) case.
To be concise, we omit the dependence on model parameters in the definition of both $\xi$ and $I$.

We now state the main theorem that characterizes Algorithm \ref{alg:sdp} and provides a sufficient condition for exact recovery on the aggregate similarity matrix of a symmetric multilayer HSBM using the SDP approach.

\begin{theorem}
\label{the:Main}
Suppose $(\bsigma, \bA) \sim \hsbm(N,M,(d_m),(\alpha_{\bt}^{(m)}))$, and let $\bW$ be the aggregate similarity matrix of $\bA$. When $I > 1$, Algorithm~\ref{alg:sdp} with $\bW$ and $s = \sgnxi$ as inputs achieves exact recovery as defined in \eqref{def:exact_recovery}.
\label{thm:main}
\end{theorem}
The proof of Theorem~\ref{the:Main} is provided in Section \ref{sec:analysis}. Taking $M=1$ with parameters
\begin{equation*}
\alpha_{(r,d-r)}
\weq
\begin{cases}
 \alpha & \text{ for } r=0 \text{ and } r=d \\
 \beta & \text{ for } 1 \le r \le d-1
\end{cases}
\end{equation*}
reduces to a homogeneous model that has been studied earlier in the
assortative case with $\xi = (d-1)(\alpha - \beta) > 0$.  In this setting Kim, Bandeira, and Goemans \cite{Kim_Bandeira_Goemans_2018-07-08} showed that the SDP algorithm does not achieve exact recovery when $I < 1$ and   Gaudio and Joshi \cite{Gaudio_Joshi_2022-08-23} proved that the SDP algorithm achieves exact recovery when $I > 1$, as conjectured in \cite{Kim_Bandeira_Goemans_2018-07-08}.

\section{Numerical illustrations}
\label{sec:numerical}
We perform numerical simulations to demonstrate the effect of the number of observed hypergraph layers on the classification accuracy of Algorithm \ref{alg:sdp} \footnote{Source code: \url{https://github.com/kalaluusua/Hypergraph-clustering.git}}. The synthetic data is sampled from a $4$-uniform $\hsbm(50,M,4,(\alpha_{\bt}^{(m)}))$, with $1\le M \le 3$. We let the hypergraph layers be identically distributed giving $\alpha_{\bm{t}}^{(m)} =: \alpha_{\bt}$ for all $m\in [M]$. We examine four scenarios: homogeneous and assortative, homogeneous and disassortative, inhomogeneous and assortative, as well as inhomogeneous and disassortative. Table \ref{tab:probs} provides the parameter values used in each case, respectively. These values are chosen such that the expected degree, i.e. the number of hyperedges a node is incident to, is the same in both the homogeneous and the inhomogeneous cases. The associated hyperedge probabilities are computed from \eqref{eq:regime}.

\begin{table}[!htb]
\caption{\label{tab:probs} The columns with numerical entries represent the parameter values ($\alpha_{\bt}$) used in each of the four simulated scenarios.}
\centering
\begin{tabular}{crrcrr}
\toprule
& \multicolumn{2}{c}{Homogeneous} & & \multicolumn{2}{c}{Inhomogeneous} \\
\cmidrule(lr){2-3} \cmidrule(lr){5-6}
 & Assortat. & Disassort. & & Assortat.\ & Disassort. \\
\midrule
$\alpha_{(4,0)}$ & $18.8$ & $7.3$ & \hphantom{X} & $18.8$  & $4.7$ \\
$\alpha_{(3,1)}$ & $7.3$ & $18.8$ & & $9.4$ & $9.4$  \\
$\alpha_{(2,2)}$ & $7.3$ & $18.8$ & & $4.7$ & $18.8$  \\
\bottomrule
\end{tabular}
\end{table}

To evaluate the accuracy of our estimate given $\bm \sigma$, we use the Hubert-Arabie adjusted Rand index (AR) \cite{Gosgens_Tikhonov_Prokhorenkova_2021,hubert_arabie_1985}, which is a measure of similarity between two community assignments. The index is equal to $1$ when the assignments are identical, and $0$ when they are independent. For each simulated hypergraph, we also compute the classification error (CE), which we define as the fraction of misclassified nodes
$
N^{-1}\min\{ \Ham(\bm{\hat{\sigma}},\bm \sigma), \Ham(\bm{\hat{\sigma}}, -\bm \sigma)\},
$
where $\Ham$ denotes the Hamming distance. The results (averaged over five different random seed initializations) are depicted in Table \ref{tab:results}.

\begin{table}[!htb]
\setlength{\tabcolsep}{2.9pt}
\centering
\caption{Classification error (CE) and Adjusted Rand index (AR) of the community assignment estimate.}
\begin{tabular} {ccrrrrrrcrrrrrr} %
\toprule
& & \multicolumn{6}{c}{Homogeneous} & & \multicolumn{6}{c}{Inhomogeneous} \\
\cmidrule(lr){3-8} \cmidrule(lr){10-15}
& & & &\multicolumn{2}{c}{Assortat.} & \multicolumn{2}{c}{Disassort.} & & & & \multicolumn{2}{c}{Assortat.} & \multicolumn{2}{c}{Disassort.} \\
\cmidrule(lr){5-6}\cmidrule(lr){7-8}\cmidrule(lr){12-13}\cmidrule(lr){14-15}
$M$ & & $|\xi|$ & $I$ & CE  & AR  &  CE & AR & \hphantom{X} & $|\xi|$ & $I$ & CE  & AR  & CE & AR \\
\midrule
1  & & $34.5$ & $0.58$ & $0.160$  & $0.464$ & $0.184$ &$0.475$ & & 42.4 & $0.41$ & $0.052$ & $0.799$ & $0.012$ & $0.952$ \\
2  & & $68.9$ & $1.08$ & $0.024$ & $0.906$ & $0.052$ & $0.800$ & & 84.8 & $0.83$ & $0.008$ & $0.969$ & $0.000$ & $1.000$\\
3  & & $103.4$ & $1.62$ & $0.004$ & $0.984$ & $0.012$ & $0.953$ & & 127.2 & $1.24$ & $0.000$ & $1.000$  & $0.000$ & $1.000$\\
\bottomrule
\end{tabular}
\label{tab:results}
\end{table}

Based on the $I$-values, we expect that the community detection performance improves as $M$ increases and is the same for the assortative and disassortative cases. This is precisely what Table \ref{tab:results} shows. Moreover, the larger $I$-values of the homogeneous case lead us to expect an overall better performance in comparison to the inhomogeneous case. Surprisingly, this is not the case. An inspection of the $\xi$-values reveals a larger level of (dis)assortativity in the inhomogeneous case. In small to moderate hypergraph sizes, we suspect that the level of assortativity may predict the detection performance of Algorithm \ref{alg:sdp} better than the information-theoretic quantity $I$, whose effect is more profound in the asymptotic regime.

\section{Analysis of the algorithm}
\label{sec:analysis}
In this section, we provide a detailed proof of Theorem \ref{thm:main}.
We follow the procedure of \cite{Gaudio_Joshi_2022-08-23,hajek2016achieving,Kim_Bandeira_Goemans_2018-07-08} to analyze the SDP framework, and extend it to a more general model $\hsbm(N,M,(d_m),(\alpha_{\bt}^{(m)}))$ that addresses multiple layers, disassortativity, and (symmetric) inhomogeneity. 

An outline of this section is as follows. Section \ref{sec:sdpanalysis} constructs a dual certificate strategy to solve the SDP in \eqref{eq:sdp} and specializes it to the assortative and disassortative cases. Bounds on certain quantities that arise as part of this strategy are provided in Sections \ref{sec:upper_bound} and \ref{sec:lower_bound}. In Section \ref{sec:assortativity}, we comment on the assortative/disassortative nature of the model and its manifestation in our analysis. Finally, Section \ref{sec:proof} puts the parts together to complete the proof of Theorem \ref{thm:main}. 

\subsection{SDP analysis}
\label{sec:sdpanalysis}

To begin, we state a sufficient condition for optimality of Algorithm \ref{alg:sdp}. 
This is a corollary of \cite[Lemma 2.2]{Gaudio_Joshi_2022-08-23} that asserts strong duality for \eqref{eq:sdp} with $s = 1$.
\begin{lemma}
\label{lem:dual}
Fix $s \in \{\pm 1\}$. Suppose there is a diagonal matrix $\bD \in \R^{N\times N}$ and $\nu \in \R$ such that the matrix $\bS \coloneqq \bD + \nu \bm{1} \bm{1}^T - s\bW$ is positive semidefinite, its second smallest eigenvalue $\lambda_{N-1}(\bS)$ is strictly positive, and $\bS \bsigma = 0$, then $\bX^* = \bsigma \bsigma^T$ is the unique optimal solution to \eqref{eq:sdp} (with the same $s$). 
\end{lemma}
For the $\hsbm(N,M,(d_m),(\alpha_{\bt}^{(m)}))$ model with node communities $\bsigma$ and the aggregate similarity matrix $\bW$, define
\begin{equation}
 \label{eq:DefD}
 D_{ii}
 \wcoloneqq s \sum_{j} W_{ij} \sigma_i \sigma_j,
\end{equation}
where $s = \sgnxi$. With $\bD = \text{diag}(D_{ii})$, it is easy to verify that $\bS \bsigma = 0$ and, therefore, it suffices to show that 
\begin{align}
 \P\left( \inf_{\bx \bot \bsigma: \| \bx \|_2=1} \bx^T \bS \bx >0 \right)
 \weq 1-o(1)
 \label{eq:pos-semidef}
\end{align}
for Lemma \ref{lem:dual} to hold. Note that $\bS \bsigma = 0$ and \eqref{eq:pos-semidef} together ensure that the kernel of $\bS$ is the line spanned by $\bsigma$ and $\lambda_{1}(\bS) \geq \dots \geq \lambda_{N-1}(\bS) > 0$ with high probability. This in particular implies that $\bm S$ is positive semidefinite and its second smallest eigenvalue $\lambda_{N-1}(\bS)$ is strictly positive, with high probability. Using a similar methodology as in \cite[Theorem 2]{hajek2016achieving}, we obtain the following complementary lemmas for the assortative and disassortative cases, respectively.

\begin{lemma}
Let $\xi >0$. With $\bD$ defined via \eqref{eq:DefD} with $ s=\sgnxi = +1$ and $\bS \coloneqq \bD + \bm{1} \bm{1}^T - \bW$, for all $\bx \bot \bsigma$ such that $\| \bx \|_2 = 1$, we have 
\[
 \bx^T \bS \bx
 \ \geq \ \min_{i} D_{ii} - \| \bW-\E \bW \|_2,
\]
where $\E \bW$ is the expected aggregate similarity matrix conditioned on $\bsigma$.
\label{lem:Sispsd}
\end{lemma}
\begin{proof}
The expected similarity matrix for a symmetric HSBM admits the following rank-$2$ decomposition:
\begin{equation*}
 \E \bW
 = \left( \frac{w_{\rm in}+w_{\rm out}}{2}\right) \bm{1}\bm{1}^T+\left( \frac{w_{\rm in}-w_{\rm out}}{2}\right) \bsigma\bsigma^T-w_{\rm in} \bI,
\end{equation*}
where $w_{\rm in} = \E[W_{ij}|\sigma_i = \sigma_j]$, $w_{\rm out} = \E[W_{ij}|\sigma_i \neq \sigma_j]$ and $\bI$ is the $N \times N$ identity matrix. 
We can then write
\begin{align*}
 \bx^T \bS \bx
 &\weq \bx^T \bD \bx + \left(\bm{1}^T \bx\right)^2 - \bx^T(\bW - \E \bW)\bx - \bx^T \E \bW \bx \\
 &\weq \bx^T \bD \bx + \left(\bm{1}^T \bx\right)^2 - \bx^T(\bW - \E \bW)\bx\\
 &\qquad - \left(\frac{w_{\rm in}+w_{\rm out}}{2}\right) \left(\bm{1}^T\bx \right)^2
 - \left( \frac{w_{\rm in}-w_{\rm out}}{2}\right) \left(\bsigma^T\bx\right)^2 + w_{\rm in} ||\bx||_2^2.
\end{align*}
Because of the definition of the spectral norm and the facts that $\bx \bot \bsigma$, and $w_{\rm in},w_{\rm out} = \Theta(\frac{\log N}{N})$ as shown in the proof of Proposition \ref{prop:assortativity_W} (Section \ref{sec:assortativity}), we obtain
\begin{align*}
 \bx^T S \bx
 &\wge \left(\min_{i} D_{ii}\right) \| \bx \|_2^2 + \left(\bm{1}^T \bx\right)^2 \left(1-\frac{w_{\rm in}+w_{\rm out}}{2}\right)
   - \bx^T(\bW-\E \bW)\bx \\
 &\wge \min_{i} D_{ii} - \|\bW-\E \bW\|_2,
\end{align*}
which proves the lemma.
\end{proof}

\begin{lemma}
Let $\xi <0$. With $\bD$ defined via \eqref{eq:DefD} with $ s=\sgnxi = -1$ and  $\bS \coloneqq \bD + \bW$, for all $\bx \bot \bsigma$ such that $\| \bx \|_2 = 1$, we have 
\[
 \bx^T \bS \bx
 \ \geq \ \min_{i} D_{ii} - \| \bW-\E \bW \|_2 - w_{\rm in}.
\]
\label{lem:Sispsd2}
\end{lemma}
\begin{proof}
The claim follows from applying the techniques from the proof of Lemma~\ref{lem:Sispsd} on
\begin{align*}
 \bx^T \bS\bx &\weq \bx^T \bD \bx - \bx^T(\E\bW - \bW)\bx + \bx^T \E \bW \bx.%
\end{align*}
\end{proof}

\subsection{Upper bound on \texorpdfstring{$\|\bW-\E \bW\|_2$}{the distance from the expected similarity matrix}}\label{sec:upper_bound}

\noindent Let $\cE$ be the set of all node sets (hyperedges) $e \subset [N]$ having size $d$. We denote by $([N], (f_e)_{e\in \cE})$ a weighted $d$-uniform hypergraph with edge weights $(f_e)$, whose similarity matrix is a zero-diagonal matrix with off-diagonal entries $(i,j)$ given by $\sum_{e: e \ni i,j} f_e$.

\begin{lemma}[Theorem 4, \cite{lee_kim_chung_2020}]
\label{lem:lkc20}
Let $G = ([N], (f_e)_{e\in \cE})$, where a random weight $f_e\in [0,1]$ is independently assigned to each hyperedge $e\in \cE$. Denote by $\bW_{f}$ the similarity matrix of $G$. Assume that $\max_{e \in \mathcal{E}} \E [f_e] \leq \frac{c_0 \log N}{\binom{N-1}{d-1}}$. Then there exists a constant $C = C(d, c_0)>0$ such that
\begin{align*}
 \P\left(\| \bW_{f} - \E \bW_{f} \|_2
 \le C\sqrt{\log N} \right)
 &\wge 1 - O(N^{-11}).
\end{align*}
\end{lemma}

Using the lemma, for the $m$-th layer of $\hsbm(N,M,(d_m),(\alpha_{\bt}^{(m)}))$ we have that 
\begin{align*}
 \P\left(\| \bW^{(m)} - \E \bW^{(m)} \|_2
 \le C^{(m)}\sqrt{\log N} \right)
 &\wge 1 - O(N^{-11}),
\end{align*}
where $C^{(m)} = C^{(m)}(d_m, \max_r \alpha^{(m)}_{(r, d_m - r)})$.
To obtain a similar bound for the aggregate similarity matrix, we let $C = \max_m C^{(m)}$ and write
\begin{align*}
    \P\left(\| \bW - \E \bW \|_2
 \le CM\sqrt{\log N} \right)
 & \weq \P\left(\Big\| \sum_{m=1}^{M} \left(\bW^{(m)} - \E \bW^{(m)} \right)\Big\|_2
 \le CM\sqrt{\log N} \right)\\
 & \wge \P\left(\sum_{m=1}^{M} \|  \bW^{(m)} - \E \bW^{(m)} \|_2
 \le CM\sqrt{\log N} \right)\\
  & \wge \P\left(\|  \bW^{(m)} - \E \bW^{(m)} \|_2
 \le C\sqrt{\log N}  , \ \forall m \in[M]\right)\\
 & \weq \prod_{m=1}^{M}\P\left(\| \bW^{(m)} - \E \bW^{(m)} \|_2
 \le C\sqrt{\log N} \right)\\
 &\wge 1 - O(MN^{-11})\\
 &\weq 1 - O(N^{-11}).
\end{align*}

\subsection{Lower bound on \texorpdfstring{$D_{ii}$}{the diagonal entries of the Laplacian}}
\label{sec:lower_bound}

\begin{lemma}
\label{lem:lower_bound}
    Let $I>1$. Then there exists a constant $\epsilon > 0$ dependent on model parameters such that for all $i \in [N]$,
    \begin{align*}
        \P(D_{ii} \le \epsilon \log N) = o(N^{-1}).
    \end{align*}
\end{lemma}
\begin{proof}
Let $\cE_m$ denote the set of all node sets of size $d_m$.
We can write $D_{ii}$ in \eqref{eq:DefD} as
\[
 D_{ii}
 \weq s\sum_m \sum_{j: j \ne i} \sum_{e \in \cE_m: e \ni i,j} A_e^{(m)} \sigma_i \sigma_j
 \weq s \sum_m \sum_{e \in \cE_m: e \ni i} A_e^{(m)} \sum_{j \in e \setminus \{i\}}  \sigma_i \sigma_j.
\]
We will split the sum on the right based on the community profile of the node set $e \setminus \{i\}$. 
Denote by $\cT_{d_m-1}$ the set of vectors $\bt^{(m)} = (t^{(m)}_{-1},t^{(m)}_{+1})$ with nonnegative integer-valued coordinates summing up to $t^{(m)}_{-1}+t^{(m)}_{+1}=d_m-1$.  For each $\bt^{(m)} \in \cT_{d_m-1}$, denote by $\cE_{i,\bt^{(m)}}$ the collection of node sets $e$ of size $d$ such that $e \ni i$ and such that the number of nodes $j \in e \setminus \{i\}$ with community membership $\sigma_j=k$ equals $t^{(m)}_k$ for $k = \{-1, +1\}$.  Then for any node $i$ with block membership $\sigma_i=k$ and 
any $e \in \cE_{i,\bt^{(m)}}$,
\[
 \sum_{j \in e \setminus \{i\}} \sigma_i \sigma_j
 \weq t^{(m)}_k - t^{(m)}_{-k}.
\]
Therefore, for any $i$ with block membership $\sigma_i=k$, we find that
\begin{equation}    
 D_{ii}
 \weq s\sum_m \sum_{\bt^{(m)} \in \cT_{d_m-1}} \sum_{e \in \cE_{i,\bt^{(m)}}} A_e^{(m)} (t^{(m)}_k-t^{(m)}_{-k})
 \weq s \sum_m \sum_{\bt \in \cT_{d_m-1}} (t^{(m)}_k-t^{(m)}_{-k}) Y^{(m)}_{i,\bt^{(m)}},
 \label{eq:Dii_as_Y}
\end{equation}
where $Y^{(m)}_{i,\bt^{(m)}} = \sum_{e \in \cE_{i,\bt^{(m)}}} A_e^{(m)}$ equals the number of hyperedges $e$ in layer $m$ that contain $i$ and for which the $i$-excluded community profile equals $\bt^{(m)} \in \cT_{d_m-1}$.  For any such $e$,
the full community profile equals $\bt^{(m)} + \be_k$, where $\be_k$ is a basis vector for the coordinate $k\in \{-1,1\}$.  Furthermore, the size of the set $\cE_{i,\bt^{(m)}}$ equals
$R_{k,\bt^{(m)}} \coloneqq |\cE_{i,\bt^{(m)}}| = \binom{\frac{N}{2}-1}{t^{(m)}_k} \binom{\frac{N}{2}}{t^{(m)}_{-k}}$.
It follows that the random variables $Y^{(m)}_{i,\bt^{(m)}}$ are mutually independent and binomially distributed according to
$Y^{(m)}_{i,\bt^{(m)}} \sim \bin(R_{k,\bt^{(m)}}, p_{\bt^{(m)} + \be_k}^{(m)})$.
Fix $\lambda \ge 0$. By independence and the inequality $1-x \le e^x$, we find that the moment-generating function of $D_{ii}$ is bounded by
\begin{align*}
 \E \left[e^{-\lambda D_{ii}}\right]
 &\weq \prod_{m} \prod_{\bt^{(m)} \in \cT_{d_m-1}} \E\left[e^{-s\lambda (t^{(m)}_k - t^{(m)}_{-k})
 Y_{i,\bt^{(m)}}^{(m)}} \right] \\
 &\weq \prod_m \prod_{\bt^{(m)} \in \cT_{d_m-1}}
 \left[1-p_{\bt+\be_k}^{(m)} \left(1-e^{-s\lambda(t^{(m)}_k - t^{(m)}_{-k})}\right)\right]^{R_{k,\bt^{(m)}}} \\
 &\wle \prod_{\bt \in \cT_{d_m-1}} \prod_m \exp
 \left[-R_{k,\bt^{(m)}} p_{\bt^{(m)}+\be_k}^{(m)}
 \left( 1-e^{-s\lambda (t^{(m)}_k - t^{(m)}_{-k})} \right) \right].
\end{align*}
Using the bounds $\left(1-\frac{j}{n}\right)^j \frac{n^j}{j!} \le \binom{n}{j} \le \frac{n^j}{j!}$
and the scaling assumption \eqref{eq:regime}, 
we find that 
\begin{equation}
    R_{k,\bt^{(m)}} p_{\bt^{(m)}+\be_k}^{(m)} = (1 + o(1)) 2^{-(d_m-1)}\binom{d_m-1}{t^{(m)}_{-k}}\alpha_{\bt^{(m)}+\be_k}^{(m)} \log N.
\label{eq:Rp}
\end{equation}

We conclude that
\[
\E \left[e^{-\lambda D_{ii}}\right]
 \wle e^{-(1+o(1)) \psi_k(s\lambda) \log N},
\]
where, for $x\in \R$,
\begin{equation}
 \psi_k(x)
 \wcoloneqq  \sum_m 2^{-(d_m-1)}\sum_{\bt \in \cT_{d_m-1}} \binom{d_m-1}{t^{(m)}_{-k}}\alpha_{\bt^{(m)}+\be_k}^{(m)}\left( 1-e^{-x (t^{(m)}_k - t^{(m)}_{-k})} \right).\nonumber
\end{equation}
For the inner summation, taking $t^{(m)}_{-k} = r$, we have that $t^{(m)}_k = d_m-1-r$ and $\alpha_{\bt^{(m)}+\be_k}=\alpha_{(r,d_m-r)} = \alpha_{(d_m-r,r)}$, thus giving 
\begin{equation}
 \psi_k(x)
 \weq  \sum_{m} 2^{-(d_m-1)}\sum_{r=0}^{d_m-1} \binom{d_m-1}{r}\alpha_{(r,d_m-r)}^{(m)}\left(1-e^{-x\left(d_m - 1 - 2r\right)}\right) \ \eqqcolon \ \psi(x).
\nonumber
\end{equation}
Note that the above expression is independent of the community of node $i$ owing to the symmetry inherent in our model.
Markov's inequality applied to the random variable $e^{-\lambda D_{ii}}$ then implies that for any $\epsilon > 0$,
\begin{equation}
 \label{eq:DiiBound}
 \P(D_{ii} \le \epsilon \log N)
 \wle e^{\lambda \epsilon \log N} \E e^{-\lambda D_{ii}}
 \wle N^{\lambda \epsilon - (1+o(1))\psi(s\lambda)}.
\end{equation}
We note that $\psi(x)$ is a concave function with $\psi(0) = 0$ and
\begin{equation}
 \psi'(0)
 \weq  \sum_m 2^{-(d_m-1)}\sum_{r = 0}^{d_m-1} \binom{d_m-1}{r}(d_m-1-2r)\alpha_{(r, d_m - r)}^{(m)} = \xi,
 \nonumber
\end{equation}
where $\xi$ is the assortativity defined by \eqref{eq:Assortativity}.
Letting $s = \sgnxi$, it follows that
$$
 I \coloneqq \sup_{x \in \R} \psi(x)
 = \sup_{\lambda \ge 0} \psi(s\lambda),
$$
where $I$ is the information quantity defined by \eqref{eq:info_metric_layers}. Given $d_m\geq 2$, we note that $-(d_m - 1 - 2r)$ is positive for at least one $0\leq r \leq d_m-1$, and the corresponding term in $\psi(x)$ decreases to $-\infty$ as $x$ increases to $\infty$. On the other hand, $-(d_m - 1 - 2r)$ is negative for at least one $0\leq r \leq d_m-1$, and the corresponding term decreases to $-\infty$ as $x$ decreases to $-\infty$. Moreover, all of the terms are bounded from above. It follows that $\psi(x)$ attains its supremum on $\R$. If we assume that $I > 1$
and choose a small enough $\epsilon > 0$, then \eqref{eq:DiiBound} implies that
\begin{equation}
\P(D_{ii} \le \epsilon \log N) \wle N^{\lambda^* \epsilon - (1+o(1))I}
\weq o(N^{-1}),
 \nonumber
\end{equation}
where $\lambda^* = \arg\max_{\lambda \ge 0} \psi(s\lambda)$.
\end{proof}

\subsection{Assortativity}
\label{sec:assortativity}

In this section, we provide an interpretation of assortativity in terms of the aggregate similarity matrix $\bW$ in the asymptotic regime. This is stated in Proposition \ref{prop:assortativity_W}, which we prove below.

\begin{proof}[Proposition \ref{prop:assortativity_W}]
    First, we note that for $k\in \{-1,1\}$ and $i\neq j$
     $$w_{\rm{in}}
 = \E[ W_{ij} \mid \sigma_i = k, \sigma_j = k]
 = \sum_{m} \sum_{\bt^{(m)}\in \cT_{d_m-2}}
 \binom{N/2-2}{t^{(m)}_{k}} \displaystyle \binom{N/2}{t^{(m)}_{-k}} p_{\bt^{(m)}+\be_k+\be_{k}}^{(m)},$$
 where $\cT_{d_m-2}$ is defined as in Section \ref{sec:lower_bound}, and 
 \begin{align*}
 w_{\rm{out}}
 = \E[ W_{ij} \mid \sigma_i = k, \sigma_j = -k] = \sum_{m}\sum_{\bt^{(m)}\in \cT_{d_m-2}} \hspace{-0.4mm}\binom{N/2-1}{t^{(m)}_{k}} \displaystyle \binom{N/2-1}{t^{(m)}_{-k}} p_{\bt^{(m)}+\be_k+\be_{-k}}^{(m)}.
\end{align*}
Applying the bounds $\left(1-\frac{j}{n}\right)^j \frac{n^j}{j!} \le \binom{n}{j} \le \frac{n^j}{j!}$
and the scaling assumption \eqref{eq:regime}, 
\begin{align*}
 w_{\rm{in}} &= \frac{\log N}{N}\cdot \sum_{m}\frac{(1 + o(1))(d_m-1)}{2^{d_m-1}} \sum_{\bt^{(m)}\in \cT_{d_m-2}}\binom{d_m-2}{t^{(m)}_{-k}} \alpha_{\bt^{(m)}+\be_k+\be_{k}}^{(m)} 
 = \Theta\left(\frac{\log N}{N}\right).
\end{align*}
Similarly, $w_{\rm{out}} = \Theta(\log N/N)$. By \eqref{eq:DefD}, for two communities of equal size we have
\begin{align}
\E D_{ii}
&= \sum_{j\neq i: \sigma_i = \sigma_j} w_{\rm{in}} - \sum_{j: \sigma_i \neq \sigma_j} w_{\rm{out}}%
 = \frac{N}{2}\left(w_{\rm{in}} - w_{\rm{out}}\right)  - o(1). \label{eq:diag_i}
\end{align}

Using (\ref{eq:Dii_as_Y}) and (\ref{eq:Rp}), the expected value of $D_{ii}$ can also be written as
\begin{align}
    \E D_{ii}
    &= (1 + o(1)) \xi\log N
  > 0 \nonumber
\end{align}

which combined with \eqref{eq:diag_i} implies the statement of the proposition.
\end{proof}

\subsection{Proof of Theorem \ref{thm:main}}
\label{sec:proof}
Lemma \ref{lem:lower_bound} shows that $D_{ii} \le \epsilon \log N$ with probability $o(N^{-1})$. Taking union bound over $i$, we obtain $\min_{i\in [N]} D_{ii} \le \epsilon \log N$ with probability $o(1)$. By Lemma \ref{lem:lkc20}, $\| \bW - \E \bW \|_2 \le CM\sqrt{\log N}$ with probability $1 - O(N^{-11})$. Moreover, $w_{\rm in} = \Theta(N^{-1}\log N)$ as shown in the proof of Proposition \ref{prop:assortativity_W}. By Lemmas \ref{lem:Sispsd} and \ref{lem:Sispsd2}, we then have $\bx^{\text{T}} \bS \bx \geq \epsilon \log N - CM\sqrt{\log N} - N^{-1}\log N> 0$ with probability $o(1)$ for all $\bx \bot \bsigma$ such that $\| \bx \|_2 = 1$. Application of Lemma \ref{lem:dual} then concludes the proof.

\section{Conclusions}
\label{sec:conclusion}
In this work, we motivated and described the non-uniform multilayer inhomogeneous HSBM. We studied the problem of exact community recovery for the model using an SDP approach and the aggregate similarity matrix. For the symmetric case, our analysis provided a sufficient condition in terms of the information quantity $I$ for community recovery. The generality of our model allows us to recover the sufficient conditions for some earlier models proposed in the literature.

Our treatment of the problem brings to the fore numerous related questions which are listed below:
\begin{itemize} \setlength{\itemsep}{2mm}

\item The assumption of symmetry on the parameters could be relaxed to make the hyperedge probabilities depend on the community labels. Additionally, it could be worthwhile to investigate asymmetry brought about by an imbalance in the community sizes.
\item This work provides sufficient conditions for exact recovery based on the SDP approach. Necessary conditions for the multilayer HSBM model with the knowledge of the similarity matrix can be obtained using a methodology similar to \cite{Kim_Bandeira_Goemans_2018-07-08} which will be addressed in a future publication. 
   
\item In this paper, the number of layers, $M$, is taken to be a constant independent on $N$. However, we expect that the analysis goes through when $M$ grows slowly with $N$.

\item The analysis of the SDP algorithm used here relies on the fact that there are just two communities. Extensions to a larger number of communities is a question worthy of investigation.

\end{itemize}

\bibliographystyle{splncs04}
\bibliography{main.bib}

\end{document}